\documentclass[11pt]{amsart}
\usepackage{latexsym,amssymb,amsmath}
\usepackage[colorlinks=true, linkcolor=blue, anchorcolor=black, citecolor=blue, filecolor=blue, menucolor= blue, urlcolor=black]{hyperref}
\usepackage{graphicx}
\usepackage{float}
\usepackage{faktor}
\usepackage[all]{xy}
\usepackage{tikz}

\newtheorem{theorem}{Theorem}[section]
\newtheorem{lemma}[theorem]{Lemma}
\newtheorem{proposition}[theorem]{Proposition}

\theoremstyle{definition}
\newtheorem{definition}[theorem]{Definition}

\newtheorem{remark}[theorem]{Remark}
\newtheorem{example}[theorem]{Example}

\newtheorem{theoremx}{Theorem}
\newtheorem{definitionmx}[theoremx]{Definition}

\begin{document}

%%%%%%%%%%%%%%%%%%%%%%%%%%%%%%%%%%%%%%%%%%%%%%%%%%%%%%%%%%%%%%%%%%%%%

\title{The edge code of hypergraphs.}

\author[D. Jaramillo-Velez]{Delio Jaramillo-Velez}
\address{
Department of Electrical Engineering\\
Chalmers University of Technology\\
SE-41296 Gothenburg, Sweden.
}
\email{delio@chalmers.se}


%\keywords{}
%\subjclass[2010]{Primary 13F20; Secondary 05C22, 05E40, 13H10.}

%\maketitle

\begin{abstract}
Given a hypergraph $\mathcal{H}$, we introduce a new class of evaluation toric codes called edge codes derived from $\mathcal{H}$. We analyze these codes, focusing on determining their basic parameters. We provide estimations for the minimum distance, particularly in scenarios involving $d$-uniform clutters. Additionally, we demonstrate that these codes exhibit self-orthogonality. Furthermore, we compute the minimum distances of edge codes for all graphs with five vertices.
\end{abstract}

\dedicatory{Dedicated to Professor Sudhir R. Ghorpade on the occasion of his sixtieth birthday.}

\maketitle

%\setcounter{secnumdepth}{1}
%\setcounter{tocdepth}{1}
%\tableofcontents

\section{Introduction}\label{intro-section}

The theory of evaluation codes, initially proposed as a generalization of one-point geometric Goppa codes by T. H$\o$holdt, J. van Lint, and R. Pellikaan \cite{Pelikaan-Van}, has been presented through various versions, including affine variety codes, order domain codes, or simply evaluation codes \cite{Jara_Pinto_Villa, Olav, Galindo-Munuera}. In this work, we introduce a new family of evaluation codes utilizing the definition and approach proposed in \cite{Jara_Pinto_Villa}.

Let $S=K[t_1,\ldots,t_s]=\bigoplus_{d=0}^\infty S_d$ denote a polynomial ring over a finite field $K:=\mathbb{F}_q$ with the standard grading. Given a geometric object $X=\{P_1,\ldots,P_m\}$ in the affine space $\mathbb{A}^s:=K^s$, an evaluation code $\mathcal{L}_{X}$ on $X$ is defined as the image of a linear subspace $\mathcal{L}\subseteq S$ under the evaluation map

\begin{equation}\label{square-free_eval}
{\rm ev}\colon \mathcal{L}\rightarrow K^{m},\quad
f\mapsto\left(f(P_1),\ldots,f(P_m)\right).
\end{equation}

%The theory of evaluation codes was proposed as generalization of one-point geometric Goppa codes by T. H$\o$holdt, J. van Lint, and R. Pellikaan \cite{Pelikaan-Van}. It has been presented through different version, affine variate codes, order domain codes or simply evaluation codes \cite{Jara_Pinto_Villa, Olav,Galindo-Munuera}. We introduce a new family of evaluation codes using the definition and the approach proposed in \cite{Jara_Pinto_Villa}. Let $S=K[t_1,\ldots,t_s]=\bigoplus_{d=0}^\infty S_d$ be a polynomial ring over a finite field $K=\mathbb{F}_q$ with the standard grading. Given a geometric object $X=\{P_1,\ldots,P_m\}$ in the affine space $\mathbb{A}^s=K^s$, an evaluation code $\mathcal{L}_{X}$ on $X$ is the image of a linear subspace $\mathcal{L}\subseteq S$ under the evaluation map 
%\begin{equation}\label{square-free_eval}
%{\rm ev}\colon \mathcal{L}\rightarrow K^{m},\quad
%f\mapsto\left(f(P_1),\ldots,f(P_m)\right).
%\end{equation}

To define families of evaluation codes, it is necessary to specify both the linear subspace $\mathcal{L}$ and the geometric object $X$ in the definition (\ref{square-free_eval}). Notable examples include:
\begin{itemize}
\item Affine Cartesian codes of degree $d$: Here, $\mathcal{L}$ is the linear space of polynomials of degree at most $d$, and $X$ is a Cartesian set \cite{Hiram-rent-villa}.
\item Monomial Cartesian codes: These are evaluation codes with $\mathcal{L}$ generated by a set of monomials, while $X$ is a Cartesian set \cite{lopez-Mate-sopro}.
\item Toric codes over a $d$-hypersimplex $\mathcal{P}$: These codes are evaluation codes on the affine torus $T:=(K^*)^s$, where $K^*$ is the multiplicative group of the field $K$. In this case, $\mathcal{L}=KV_d$ is the $K$-linear subspace of $S_d$ spanned by the set $V_d$ of all $t^a:=t_1^{a_1}\dots t_s^{a_s}$ such that $a\in\mathcal{P}\cap\mathbb{Z}^s$ \cite{Jara_Pinto_Villa}, \cite[p. 84]{Sturmfels}.
\item Squarefree evaluation codes: These are toric codes with the space of polynomials generated by squarefree monomials \cite{Jara_Pinto_Villa}.
\end{itemize}

Several strategies have been devised to harness the combinatorial structure of hypergraphs for defining both non-linear and linear codes, as well as for studying their fundamental parameters such as length, dimension, and minimum distance. Inspired by the concept of expander codes \cite{Zemor}, A. Barg and G. Zemor introduced the notion of a ``hypergraph code", where the coordinates of the codewords correspond one-to-one with the edges of a uniform $t$-partite, $\Delta$-regular hypergraph. Similarly, Y. Bilu and S. Hoory utilized the structure of $t$-uniform $t$-partite $\Delta$-regular hypergraphs for efficient decoding of binary codes on hypergraphs \cite{Bilu-Hoory}.

For evaluation codes, graphs have also been employed to define novel families of codes. Given a projective toric set of points $X$ parameterized by the edges of a graph $G$, C. Renter\'ia et al. introduced the concept of a ``parameterized code of order $d$ over $G$," which constitutes affine Cartesian codes on the set $X$ \cite{Neves-Pinto, ren-simis-villa}. Furthermore, based on the incidence matrix of a signed graph $G_{\sigma}$, J. Mart\'inez-Bernal et al. introduced the ``incidence matrix code" as a linear code generated by the row vectors of the incidence matrix of $G_{\sigma}$ \cite{Bernal-Bucio-Villarreal}.

%Several strategies have been  introduced to leverage the  combinatorial structure of hypergraphs to define no-linear and linear codes, as well to study their basic parameters, length, dimension, and minimum distance. Inspired by the notion of expander codes \cite{Zemor}, A. Barg, and G. Zemor introduce {\it hypergraph code} as binary codes where the coordinates of the codewords are in one-to-one correspondence with the edges of a uniformt $t$-partite, $\Delta$-regular hypergraph. Similarly, Y. Bilu , S. Hoory leverage from the structure $t$-uniform $t$-partite $\Delta$-regular hypergraphs to efficiently decode binary codes on hypergraphs \cite{Bilu-Hoory}. For the case of evaluation code, graphs are also used to define new families of codes. Given a projective toric set of point $X$ parameterized by the edges of a graph $G$, C. Renter\'ia et al. introduced {\it parameterized code of order $d$ over $G$}, as a affine Cartesian codes on the set $X$ \cite{Neves-Pinto, ren-simis-villa}. Based on the incidence matrix of a signed graph $G_{\sigma}$, J. Mart\'inez-Bernal et al. introduced incidence matrix code as a linear code generated by the row vectors of the incidence matrix of $G_{\sigma}$ \cite{Bernal-Bucio-Villarreal}. 

Given a hypergraph $\mathcal{H}$, we introduce a novel family of evaluation codes defined on the affine torus $T$. In contrast to previous research on graphs and evaluation codes, our approach involves parameterizing the set of polynomials by the edges of the hypergraph, while fixing the set of points. Additionally, our methodology extends beyond traditional graphs to encompass hypergraphs, which produce a class of codes that in some ways generalize toric codes over hypersimplices. 
%We compute the basic parameters of these evaluation codes, paving the way for further exploration and utilization in various contexts.

%Given a hypergraph $\mathcal{H}$ we propose a new family of evaluation codes, and compute their basic parameters on the affine torus $T$. Different from previous works on graphs and evaluation codes we  parameterize the set of polynomials by the edges and we fix the set of points, additionally we consider hypergraphs and not only graphs.  

Let $\{e_1,\dots,e_n\}$ represent the set of edges of the hypergraph $\mathcal{H}$ with $s$ vertices. We denote $E(\mathcal{H})=\{f_1,\dots,f_n\}$, where each $f_i:=\prod_{j\in e_i}t_j$ is defined as the product of variables indexed by $j$ belonging to edge $e_i$ in the hypergraph $\mathcal{H}$, thereby indicating the set of monomials generated by its edges.

\begin{definitionmx}

The edge toric code or edge code $C_{\mathcal{H}}$ of $\mathcal{H}$ is defined as the image ${\rm ev}(KE(\mathcal{H}))$ of the $K$-linear space generated by the monomials in $E(\mathcal{H})$, denoted as $KE(\mathcal{H})$, under the evaluation map specified in (\ref{square-free_eval}).

\end{definitionmx}

A clutter $\mathcal{C}$ is defined as a hypergraph in which no edges are contained within other edges. Additionally, $\mathcal{C}$ is termed $d$-uniform if every edge within it has a size of $d$. In our study, we focus on estimating the minimum distance of this new class of codes for specific cases, leveraging the properties of $d$-uniform clutters.

\begin{theoremx}[Theorem \ref{clutter_theorem}]\label{clutter_theorem_intro}
Let $\mathcal{C}$ be an $d$-uniform clutter, and $q\geq 3$. The minimum distance of the edge code $C_{\mathcal{C}}$ is given by 
$$
\delta(C_{\mathcal{C}})=\begin{cases}
(q-2)^d(q-1)^{s-d}&\mbox{ if } d\leq s/2, \mbox{ and } \mathfrak{C}_{d-1}(J_{d\times 2})\subseteq \mathcal{C}\\
(q-2)^{s-d}(q-1)^{d}&\mbox{ if }s/2 < d < s, \mbox{ and } \mathfrak{C}_{s-d-1}(J_{(s-d)\times 2})\subseteq \mathcal{C^*}\\
\end{cases}
$$
\end{theoremx}

%\begin{theoremx}[Theorem \ref{clutter_theorem}]\label{clutter_theorem_intro}
%Let $\mathcal{C}$ be a homogeneous clutter of degree $d$, and $q\geq 3$. The minimum distance of the edge code $C_{\mathcal{C}}$ is given by 
%$$
%\delta(C_{\mathcal{C}})=\begin{cases}
%(q-2)^d(q-1)^{s-d}&\mbox{ if } d\leq s/2, \mbox{ and $\mathcal{C}$ has a $2d$-cycle.}\\
%(q-2)^{s-d}(q-1)^{d}&\mbox{ if }s/2 < d < s, \mbox{ and $\mathcal{C}^{*}$ has a $2(s-d)$-cycle}.\\
%\end{cases}
%$$
%\end{theoremx}

%These results generalize the computations for toric codes over hypersimplexes.
Furthermore,  we  compute the minimum distance of the edge code  $C_{\mathcal{H}}$ for a special hypergraph $\mathcal{H}$ that is not a clutter.  

\begin{theoremx}[Theorem \ref{relative_code}]\label{relative_code_intro}
Let  $1\leq d_2\leq d_1$ two integers, and $q\geq 3$. If the hypergraph $\mathcal{H}$ contains all the edges of size $r$ with $d_2\leq r\leq d_1$, the minimum distance of the edge code is given by 
$$
\delta(\mathcal{C}_{d_{1}d_{2}}) = \begin{cases} (q-2)^{d_{1}}(q-1)^{s-d_1}\hbox{ if } d_1+d_2\leq s\\
(q-2)^{s-d_{2}}(q-1)^{d_2}\hbox{ if }d_1+d_2> s.\\
\end{cases}
$$
\end{theoremx}

Moreover, we establish the self-orthogonality property of edge codes, as demonstrated in Theorem \ref{dual_code}. Additionally, we provide a comprehensive analysis of weight distribution for edge codes of trees, 

\begin{theoremx}[Theorem \ref{theo-weight-dist}]\label{theo-weight-dist-intro}
Let $C_G$ be the edge code of connected tree $G$. For $1\leq t\leq (s-1)/2,$ the $t$-th Hamming weight of $C_G$ is given by
$$
(q-1)^{s}-(q-1)^{s-1}+(q-1)^{s-2}+\cdots+(-1)^{2t+1}(q-1)^{s-(2t-1)}.
$$
\end{theoremx}
Finally, we compute the minimum distance of the edge code $C_{G}$ for all connected graphs $G$ with five vertices in Table \ref{five_vertices}.

The rest of the paper is organized as follows. In Section \ref{prelimina}, we present the results and definition for the proof of our findings. Section \ref{edge_coding} delves into the computation of weights and basic parameters of edge codes for specific families of hypergraphs. Finally, Section \ref{Examples} presents illustrative examples.

%%%%%%%%%%%%%%%%%%%%%%%%%%%%%%%%%%%%%%%%%%%%%%%%%%%%%%%%%
\section{Preliminaries}\label{prelimina}
%%%%%%%%%%%%%%%%%%%%%%%%%%%%%%%%%%%%%%%%%%%%%%%%%%%%%%%%% 
This section presents basic definitions and results that set the foundation for the further results. Specifically, we revisit the concept of affine Hilbert functions and the notion of the degree of an ideal. These concepts are crucial in determining the maximum number of common zeros of a system of polynomials.

For a more comprehensive understanding of these concepts and results, readers are referred to~\cite{MacWilliams-Sloane, mon-alg}.

Consider the following notation: 
\begin{itemize}
\item {\bf Krull dimension:} Let $I$ be an ideal of $S=K[t_1,\ldots,t_s]$. The Krull dimension of $S/I$ is denoted by $\dim(S/I)$. We say that the ideal $I$ has a dimension $k$ if $\dim(S/I)$ is equal to $k$.
\item {\bf Height:} The height of $I$, denoted and defined by $\text{ht}(I):=s-\dim(S/I)$.
\item {\bf Affine Hilbert function:} The $K$-linear space of polynomials in $S$  (resp. $I$) of degree at most $d$ is denoted by $S_{\leq d}$ (resp. $I_{\leq d}$). The function $H_I^a(d):=\dim_K(S_{\leq d}/I_{\leq d}),\ \ \ d=0,1,2,\ldots$ is called the affine Hilbert function of $S/I$.
\end{itemize}
%%Let $I$ be an ideal of $S=K[t_1,\ldots,t_s].$ The Krull dimension of $S/I$ is denoted by $\dim(S/I)$. We say that the ideal $I$ has {\it dimension} $k$ if $\dim(S/I)$ is equal to $k$. The \textit{height} of $I$ is denoted and defined by ${\rm ht}(I):=s-\dim(S/I)$. The $K$-linear space of polynomials in $S$  (resp. $I$) of degree at most $d$ is denoted by $S_{\leq d}$ (resp. $I_{\leq d}$). The function
%$$
%H_I^a(d):=\dim_K(S_{\leq d}/I_{\leq d}),\ \ \ d=0,1,2,\ldots,
%$$
%is  called the \textit{affine Hilbert function} of $S/I$. 
According to Hilbert's theorem \cite{CR-Mat}, there exists a unique polynomial $h^a_I(z)=\sum_{i=0}^{k}a_iz^i\in \mathbb{Q}[z]$ of degree $k$ such that $h^a_I(d)=H_I^a(d)$ for sufficiently large $d$. We adopt the convention, the degree of the zero polynomial is $-1$. The integer $k!a_k$, denoted $\text{deg}(S/I)$, is referred to as the \textit{degree} of $S/I$. In the case where $k=0$, $H_I^a(d)=\text{deg}(S/I)=\dim_K(S/I)$ for $d\gg 0$. It is worth noting that the degree of $S/I$ is positive if $ I\subsetneq S$, and $0$ otherwise.\\
%By a Hilbert theorem \cite{CR-Mat}, there is a unique polynomial $h^a_I(z)=\sum_{i=0}^{k}a_iz^i\in  \mathbb{Q}[z]$ of degree $k$ such that $h^a_I(d)=H_I^a(d)$ for $d\gg 0$. By convention the degree of the zero polynomial is $-1$. The integer $k!\, a_k$, denoted ${\rm deg}(S/I)$, is called the \textit{degree} of $S/I$.  If $k=0$, then $H_I^a(d)=\deg(S/I)=\dim_K(S/I)$ for $d\gg 0$. Note that the degree of $S/I$ is positive if $ I\subsetneq S$ and is $0$ otherwise. 
The \textit{affine variety} of a finite subset $F$ of $S$ with respect to the set of points $X=\{P_1,\ldots,P_m\}$, denoted by $V_X(F)$, is defined as the set of all points $P \in X$ such that $f(P)=0$ for all $f\in F$. The {\it colon ideal} is defined as  $$(I\colon(F)):=\{g\in S\, \vert\, gF\subseteq I \}.$$ The colon ideal is an important tool for determining whether the affine variety $V_X(F)$ is non-empty (see Lemma~\ref{vila-delio-feb27-20}).

Now, we recall some important properties related to the relationship between the degree of a vanishing ideal and the number of elements of an affine variety:

\begin{lemma}{\cite[p.~389]{rubiano}}\label{primdec-ixx} 
Let $X$ be a finite subset of
$\mathbb{A}^s$, let $P$ be a point in $X$, $P=(p_1,\ldots,p_s)$, and
let $I_{P}$ be the vanishing ideal 
of $P$. Then $I_P$ is a maximal ideal of height $s$, 
\begin{equation*}
I_P=(t_1-p_1,\ldots,t_s-p_s),\ \deg(S/I_P)=1, 
\end{equation*}
and $I(X)=\bigcap_{P\in X}I_{P}$ is the primary 
decomposition of $I(X)$.  
\end{lemma}

\begin{lemma}{\rm(cf. \cite[Lemma~3.3]{bernal_pitones_Villarreal})}\label{apr5-19}
Let $d_1,\ldots,d_s$ be positive integers and let $L$ be the ideal of $S$ generated by
$t_1^{d_1},\ldots,t_{s}^{d_s}$. If $t^a=t_1^{a_1}\cdots t_s^{a_s}$ is
not in $L$, then
$$
\deg(S/(L,t^a))=d_1\cdots d_s-(d_1-a_1)\cdots(d_s-a_{s}).
$$
\end{lemma}

\begin{lemma}{\rm \cite[Lemma 2.5]{Jara_Pinto_Villa}}\label{vila-delio-feb27-20} Let $X$ be a finite subset of 
$\mathbb{A}^s$ over a field $K$ and let $F=\{f_1,\ldots,f_r\}$ be a
set of polynomials of $S$. Then, the following conditions are equivalent.
\begin{enumerate}
\item[(a)] $(I(X)\colon(F))=I(X)$.
\item[(b)] $V_X(F)=\emptyset$.
\item[(c)] $(I(X),F)=S$.
\end{enumerate}
\end{lemma}

%We remain the notions of standard monomials, footprint, and Gr\"obner basis, which play a fundamental role in the devolpment of our results.
%Let $\prec$ be a monomial order, ${\rm in}_\prec(I)$ denotes the initial ideal of $I$. A monomial $t^{a}:=t_1^{a_1}\dots t_s^{a_s}$ is called standard if $t^{a}\not\in{\rm
%in}_\prec(I)$, $\Delta_\prec(I)$ is the set of standard monomials of $S/I$, which is called the footprint of $I$.  
%
%The image of $\Delta_\prec(I)$, under the canonical 
%map $S\mapsto S/I$, $x\mapsto \overline{x}$, is a basis of $S/I$ as a
%$K$-vector space \cite[Proposition~6.52]{Becker-Weispfenning}.From this we can infer that $H_{I}^{a}(d)$ is the number of standard monomials of $S/I$ of degree at most $d$. For the set of polynomials $F\subseteq S$ we set ${\rm in}_{\prec}(F):=\lbrace {\rm in}_{\prec}(f): f\in F\rbrace$.  Finally, a subset $\mathcal{G}=\{g_1,\ldots, g_n\}$ of $I$ is called a 
%{\it Gr\"obner basis\/} of $I$ if ${\rm in}_\prec(I)=({\rm in}_\prec(g_1),\ldots,{\rm in}_\prec(g_n))$. 

\begin{theorem}{\rm \cite[Theorem 2.12]{Jara_Pinto_Villa}}\label{degree-initial-theorem}
Let $X$ be a finite subset of $\mathbb{A}^{s}$, let $I=I(X)$ be the vanishing ideal of $X$, and let $\prec$ be a monomial order. If $F$ is a finite set of polynomials of $S$ and $(I:(F))\neq I$, then
$$
|V_{X}(F)|={\rm deg}(S/\left( I,F\right) )\leq{\rm deg}(S/\left( {\rm in}_{\prec}(I),{\rm in}_{\prec}(F)\right) )\leq {\rm deg}(S/I )=|X|,
$$
and ${\rm deg}(S/\left( I,F\right) )<{\rm deg}(S/I )$ if $(F)\not\subseteq I$.
\end{theorem}

%%Parameters of a code
We retain the notions of standard monomials, footprint, and Gr\"obner basis, which play a fundamental role in the development of our results.

Let $\prec$ be a monomial order, and ${\rm in}_\prec(I)$ denote the initial ideal of $I$. A monomial $t^{a}=t_1^{a_1}\cdots t_s^{a_s}$ is called standard if $t^{a}\not\in{\rm in}_\prec(I)$. The set of standard monomials of $S/I$, denoted by $\Delta_\prec(I)$, is called the footprint of $I$.

The image of $\Delta_\prec(I)$ under the canonical map $S\mapsto S/I$, $x\mapsto \overline{x}$, forms a basis of $S/I$ as a $K$-vector space \cite[Proposition~6.52]{Becker-Weispfenning}. Consequently, $H_{I}^{a}(d)$ represents the number of standard monomials of $S/I$ of degree at most $d$.

For a set of polynomials $F\subseteq S$, we define ${\rm in}_{\prec}(F):=\lbrace {\rm in}_{\prec}(f): f\in F\rbrace$.

Finally, a subset $\mathcal{G}=\{g_1,\ldots, g_n\}$ of $I$ is called a {\it Gr\"obner basis} of $I$ if ${\rm in}_\prec(I)=({\rm in}_\prec(g_1),\ldots,{\rm in}_\prec(g_n))$.

We turn our attention to the notion of weights of an $[m,k]$-{\it linear
code} $C\subseteq K^m$ of {\it length} $m$ and {\it dimension} $k=\dim_K(C)$. Given a $r$-subcode $D$ of $C$ (that is, $D$ is an $r$ dimensional linear subspace of $C$), the {\it support\/} $\chi(D)$ of $D$ is the set   
$$
\chi(D):=\{i\,\vert\, \exists\, (a_1,\ldots,a_m)\in D,\, a_i\neq 0\}.
$$

%i am here

%Let $C$ be a $[m,k]$ {\it linear
%code} of {\it length} $m$ and {\it dimension} $k=\dim_K(C)$, 
%that is, $C$ is a linear subspace of $K^m$ with $k=\dim_K(C)$. 
%\quad If $r=1$, $\delta_r(C)$ is the \textit{minimum distance} of $C$ and is denoted simply by $\delta(C)$.\\
\quad The Hamming weight, or just weight, of an element $c\in C$ is the number of non-zero coordinates in $c$, or equivalent to the cardinality, $|\chi(\langle c\rangle)|$ of the support of the linear subspace generated by the codeword $c$. Let $A_i(C)$ be number of codewords of weight $i$  in $C$. The list $A_i$ for $1\leq i\leq m$  is called the weight distribution of $C$. The elements $i>0$ for which $A_i(C)\neq 0$ are called the {\it Hamming weights} or {\it next-to-weights} of $C$. The {\it minimum distance} is the first non-zero Hamming weight of $C$, and it is denoted by $\delta(C)$.

\begin{remark}\label{chi_zeros}
The mininum distance of an evaluation code has nice interpretation in terms of the number elements of an affine variety. If $D$ is a linear subspace of an evaluation code $\mathcal{L}_{X}$, then $D={\rm ev}(\mathcal{L}')$, where $\mathcal{L}'$ is a linear subspace of $\mathcal{L}$ with basis $\mathcal{B}=\{f_1,\dots,f_r\}$. Thus, 
\begin{align*}
\chi(D)=\{i\,\vert\, \exists\,f\in\mathcal{L}', f(P_i)\neq 0\},
\end{align*}
with $|\chi(D)|=|X\setminus V_{T}(\mathcal{B})|$. This gives an alternative description for the minimum distance of the evaluation code $\mathcal{L}_X$, 
\begin{align*}
\delta(\mathcal{L}_X)&=\min_{f\in\mathcal{L}}|X\setminus V_{T}(f)|\\
&=|X|-\max_{f\in\mathcal{L}}|V_{X}(f)|.
\end{align*}
Then, to compute the minimum distance of an evaluation codes it is sufficient to find a sharp bound to maximum number of zeros of a polynomial $f\in\mathcal{L}$ \cite{Sudhir_one, Sudhir_three}.
\end{remark}

In the rest of the paper, by \textit{basic parameters} of the linear code $\mathcal{L}_X$, we refer to the following numerical values:
\begin{itemize}
\item[(a)] \textit{length}: $|X|$,
\item[(b)] \textit{dimension}: $\dim_K(\mathcal{L}_X)$, and 
\item[(c)] \textit{Minimum distance}: $\delta(\mathcal{L}_X)$. 
\end{itemize}

\section{Minimum distance of edge codes}\label{edge_coding}
This section is devoted to present the computation of the basic parameters of edge codes for some special families of hypergraphs. In particular, we compute the minimum distance in the case of $d$-uniform clutters, and a weight distribution for the edge code of tree graphs. Additionally, we prove that edge codes are self-ortogonal.

Let $\mathcal{H}=(V,E)$ be a hypergraph, where $V=\{1,\dots,s\}$ is the set of vertices and $E=\{e_1,\dots,e_m\}$ is the set of edges. We start by computing the two first basic parameters of an edge code on $\mathcal{H}$, which are the length, and the dimension of the code. 

\begin{proposition}
Given a hypergraph $\mathcal{H}$ on $s$ vertices and set of edges $E$. The edge code $C_{\mathcal{H}}$ has length $(q-1)^s$ and dimension $|E|$. 
\end{proposition}
\begin{proof}
These computations are derived from the definition of an evaluation. The length of an evaluation code is given by the number of points, which in our case it is the number of points on the affine torus $(q-1)^s$. The evaluation map on (\ref{square-free_eval})
is injective, so the dimension of the code is equal to the dimension of the linear space generated by the edges of the hypergraph, which is equal to $|E|$.
\end{proof}

Now, we leverage from the following nice propertive of the torus to definite a type ``edge-complement code" of an edge code. The affine torus $T$ on $\mathbb{A}^{s}$ is a group under component-wise multiplication and the map $\sigma\colon{T}\rightarrow{T}$, $[(x_1,\ldots,x_s)]\mapsto [(x_1^{-1},\ldots,x_s^{-1})]$ is a group isomorphism. Setting $Q_i:=\sigma(P_i)$ for $i=1,\ldots,m$, we can write 
$$
{T}=\{P_1,\ldots,P_m\}=\{Q_1,\ldots,Q_m\}.
$$

The {\it Edge-removed hypergraph} $\mathcal{H}^*$ is a hypergraph with the same set of vertices $V$, but the edges are formed by removing each edge $e_i$ from the vertex set $V$, i.e. $E'=\{V\setminus e_1,\dots,V\setminus e_m\}$. The {\it Edge-removed code} $C_{\mathcal{H}^{*}}$ of a hypergraph $\mathcal{H}$ is the edge code of $\mathcal{H}^*$  on ${T}=\{Q_1,\ldots,Q_m\}$, i.e., it is the image of the evaluation map
\begin{equation}\label{square_removed_eval}
{\rm ev}\colon KE(\mathcal{H}^*)\rightarrow K^{m},\quad
f\mapsto\left(f(Q_1),\ldots,f(Q_m)\right).
\end{equation}
The following the lemma establish the relation between and edge code and its edge-removed code. This becomes important for our purposes because we can reduce some computation to the case of the edge-removed code.

\begin{lemma}\label{aster_code}
The codes $C_{\mathcal{H}}$ and $C_{\mathcal{H}^{*}}$ have the same basic parameters.
%(Still they are not equivalent)
\end{lemma}

\begin{proof}
Both codes are define over the torus, and the hypergraphs $\mathcal{H}$, and $\mathcal{H}^{*}$ have the same amount of edges, then the codes $C_{\mathcal{H}}$ and $C_{\mathcal{H}^{*}}$ have the same length and dimension. We now prove that they have the same minimum distance. For this  consider, the set of all squarefree monomials
of  $S$ generated from the edges of the hypergraph $\mathcal{H}$, $E=\{f_1,\ldots,f_n\}$. Let $E^*=\{f_1^*,\ldots,f_n^*\}$ be the set of squarefree monomials generated from the edges of the hypergraph $\mathcal{H}^{*}$, which are actually given by  
$$f_i^*:=t_1\cdots t_s/f_i,$$ 
for $i=1,\ldots,n$.  For $f\in KE$ we write $f=\sum_{i=1}^n\lambda_if_i$, $\lambda_i\in\mathbb{F}_q$ for all $i$, and 
we set $f^*:=\sum_{i=1}^n\lambda_if_i^*$. From the equalities 
\begin{eqnarray*}
t_1\cdots t_sf^*\left(\frac{1}{t_1},\ldots,\frac{1}{t_s}\right)
&=&t_1\cdots
t_s\sum_{i=1}^n\lambda_if_i^*\left(\frac{1}{t_1},\ldots,\frac{1}{t_s}\right)\\
&=&
\sum_{i=1}^n\lambda_if_i(t_1,\ldots,t_s)=f(t_1,\ldots,t_s),
\end{eqnarray*}
we get $t_1\cdots t_sf^*\left({t_1^{-1}},\ldots,{t_s^{-1}}\right)
=f(t_1,\ldots,t_s)$. Hence, setting $P_k=(p_{k,1},\ldots,p_{k,s})$,
one has 
\begin{equation}\label{feb11-20}
p_{k,1}\cdots p_{k,s}f^*(p_{k,1}^{-1},\ldots,
p_{k,s}^{-1})=p_{k,1}\cdots p_{k,s}f^*(Q_k)=f(P_k) 
\end{equation}
for $k=1,\ldots,m$. There is a commutative diagram
$$
%\begin{array}{ccc}
%\begin{array}{ccc}
%KV_d
%&\stackrel{\mathrm{ev}_d\ }{\longrightarrow}& 
%\mathcal{C}_{d}\\ 
%\Big\downarrow\rlap{$\psi$}& &\Big\downarrow\rlap{$\psi'$}\\ 
%KV_{s-d}
%&\stackrel{\mathrm{ev}_{s-d}}{\longrightarrow}& 
%\mathcal{C}_{s-d}
%\end{array}
%&
%\quad
%&
\begin{array}{ccc}
f &{\longrightarrow}& 
(f(P_1),\ldots,f(P_m))\\ 
\Big\downarrow& &\Big\downarrow\\ 
f^{*}
&{\longrightarrow}& 
(f^*(Q_1),\ldots,f^*(Q_m))
\end{array}
%\end{array}
$$
where all the maps are isomorphisms of linear spaces.
Equation (\ref{feb11-20}) implies that a polynomial $f$ vanishes at a point  $P_k$ if and only if $f^*$ vanishes at $Q_k$. Hence, by Remark \ref{chi_zeros} and the definition of $C_{\mathcal{H}^*}$ we can realize that the linear codes $C_{\mathcal{H}}$ and $C_{\mathcal{H}^*}$ have the  same minimum distance. 
\end{proof}

As previously mentioned, determining the minimum distance of evaluation codes equates to estimating the maximum number of zeros a polynomial in the linear space $\mathcal{L}$ can possess \cite{Sudhir_two,Sudhir_four}. The subsequent propositions delineate certain limits on the maximum number of zeros for a square-free polynomial. Subsequently, we employ these findings to calculate the minimum distance of an edge code.

\begin{proposition}\cite[Proposition 4.3]{Jara_Pinto_Villa}\label{squarefree-affine} Let $f$ be a squarefree 
polynomial in $S\setminus\mathbb{F}_q$ of degree at most $d$ and let $T$
be the affine torus $(\mathbb{F}_q^*)^s$ of $\mathbb{A}^s$. If $q\geq
3$ and $d\leq s$, then 
$$
|V_T(f)|\leq(q-1)^{s}-(q-2)^{d}(q-1)^{s-d}.
$$

\end{proposition}

\begin{proposition}\cite[Corollary 4.6]{Jara_Pinto_Villa}\label{bound_zeros} Let $f$ be a monomially squarefree 
homogeneous polynomial in $S_d\setminus\mathbb{F}_q$ and let $T$
be the affine torus $(\mathbb{F}_q^*)^s$ of $\mathbb{A}^s$. If $q\geq
3$ and $s/2 < d < s$, then 
$$
|V_T(f)|\leq(q-1)^{s}-(q-2)^{s-d}(q-1)^{d}.
$$
and this upper bound is sharp.
\end{proposition}

\begin{proposition}\label{zero_poly}
Let $f$ be a squarefree polynomial $S\setminus\mathbb{F}_q$ of degree $d$ such that $f=h_1\dots h_d$ with $h_i=t_k-t_l$ or $h_i=t_k-1$, and for $i\neq j$,  $h_i$, $h_j$ do not have variables in common. Then, 
$$
|V_T(f)|=(q-1)^{s}-(q-2)^{s-d}(q-1)^{d}.
$$
 
\end{proposition}

\begin{proof}
Note that $V_T(f)$ is equal to $\bigcup_{i=1}^{d_1}V_T(h_i)$, using the inclusion-exclusion  principle \cite[p.~38, Formula~2.12]{Aigner}, we get 
\begin{eqnarray*}
|V_T(f)|&=&\sum_{1\leq \ell_1\leq
d}|V_T(h_{\ell_1})|-\sum_{1\leq \ell_1<\ell_2\leq
d}|V_T(h_{\ell_1})\textstyle\cap V_T(h_{\ell_2})|\\
& &\ \ \ \ \ \ \ \ +\cdots+(-1)^{d-1}|V_T(h_1)
\textstyle\cap \cdots
\textstyle\cap  V_T(h_{d})|.
\end{eqnarray*}
\quad The variables occurring in $h_i$ and $h_j$ are disjoint for $i\neq j$, then $|V_T(h_{\ell_1})\textstyle\cap \cdots{\textstyle\cap}V_T(h_{\ell_i})|= (q-1)^{s-i}$. 
Thus, considering all possible combination, one obtains
$$
\sum_{1\leq\ell_1<\cdots<\ell_i\leq
d}|V_T(h_{\ell_1})\textstyle\cap \cdots{\textstyle\cap}V_T(h_{\ell_i})|=
\displaystyle\binom{d}{i}(q-1)^{s-i},
$$
and consequently the number of zeros of $f$ in $T$ is given by
\begin{eqnarray}\label{Notre-Dame-Cathedral}
|V_T(f)|&=&\sum_{i=1}^{d}(-1)^{i-1}\binom{d}{i}(q-1)^{s-i}=(q-1)^{s-d}
\sum_{i=1}^{d}(-1)^{i-1}\binom{d}{i}(q-1)^{d-i}\nonumber\\
&=&(q-1)^{s-d}\left[(q-1)^{d}-((q-1)-1)^{d}\right]\nonumber\\
&=&
(q-1)^{s-d}\left[(q-1)^{d}-(q-2)^{d}\right]\nonumber\\
&=&(q-1)^s-(q-2)^{d}(q-1)^{s-d}.
\end{eqnarray}
\end{proof}

To describe the combinatorial relation between the minimum distance of an edge code and its hypergraph, we use some special paths on a multipartite graph. We use the notation $G_{i_1,\dots,i_n}$ to denote a connected multipartite graph with vertex partition $V(G_{i_1,\dots,i_n})=V_1\sqcup \dots\sqcup V_n$ such that $|V_k|=i_k$ for $k=1,\dots,n$.

\begin{definition}
A {\it partite path} of length $k$ in a multipartite graph $G_{i_1,\dots,i_n}$ is a path $\{j_1,\dots,j_{k+1}\}$ such that $j_r\in V_r$. The set of  partite paths of length $k$ in $G_{i_1,\dots,i_n}$ forms a clutter and it is denoted by $\mathfrak{C}_{k}(G_{i_1,\dots,i_n})$.
\end{definition}

In particular, we use the clutter of  partite paths of the following multipartite graph, denoted by $J_{d\times2}$, with $2d$ vertices $V(J_{d\times2})=\{t_{i_1},t_{i_2}\}\sqcup \dots\sqcup \{t_{i_{2d-1}},t_{i_2d}\}$.

\begin{figure}[H]
\centering
\begin{tikzpicture}[line width=.5pt,scale=0.75]
		\tikzstyle{every node}=[inner sep=1pt, minimum width=5.5pt] 
\tiny{
\node (1) at (-1,1){$\bullet$};
\node (3) at (1,1) {$\bullet$};
\node (2) at (-1,-1) {$\bullet$};
\node (4) at (1,-1) {$\bullet$};
\node (5) at (3,1) {$\bullet$};
\node (6) at (3,-1) {$\bullet$};
\node (7) at (4,1) {};
\node (8) at (4,-1) {};
\node (9) at (4,0) {};
\node at (-1,1.3){$t_{i_1}$};
\node at (1,1.3) {$t_{i_3}$};
\node at (-1,-1.3){$t_{i_2}$};
\node at (1,-1.3) {$t_{i_4}$};
\node at (3,1.3){$t_{i_5}$};
\node at (3,-1.3) {$t_{i_6}$};
\node at (4.7,0) {$\dots$};
\draw[-,line width=1pt] (1) to (3);
\draw[-,line width=1pt] (1) to (4);
\draw[-,line width=1pt] (2) -- (3);
\draw[-,line width=1pt] (2) -- (4);
\draw[-,line width=1pt] (3) -- (5);
\draw[-,line width=1pt] (3) -- (6);
\draw[-,line width=1pt] (4) -- (5);
\draw[-,line width=1pt] (4) -- (6);
\draw[-,line width=1pt] (5) -- (7);
\draw[-,line width=1pt] (6) -- (9);
\draw[-,line width=1pt] (6) -- (8);
\draw[-,line width=1pt] (5) -- (9);
}
\end{tikzpicture}
\begin{tikzpicture}[line width=.5pt,scale=0.75]
		\tikzstyle{every node}=[inner sep=1pt, minimum width=5.5pt] 
\tiny{
\node (1) at (-1,1){$\bullet$};
\node (3) at (1,1) {$\bullet$};
\node (2) at (-1,-1) {$\bullet$};
\node (4) at (1,-1) {$\bullet$};
\node (5) at (3,1) {$\bullet$};
\node (6) at (3,-1) {$\bullet$};
\node (7) at (-2,1) {};
\node (8) at (-2,-1) {};
\node (9) at (-2,0) {};
\node at (-1,1.3){$t_{i_{2d-5}}$};
\node at (1,1.3) {$t_{i_{2d-3}}$};
\node at (-1,-1.3){$t_{i_{2d-4}}$};
\node at (1,-1.3) {$t_{i_{2d-2}}$};
\node at (3,1.3){$t_{i_{2d-1}}$};
\node at (3,-1.3) {$t_{i_{2d}}$};
%\node at (-2.5,0) {$\dots$};
\draw[-,line width=1pt] (1) to (3);
\draw[-,line width=1pt] (1) to (4);
\draw[-,line width=1pt] (2) -- (3);
\draw[-,line width=1pt] (2) -- (4);
\draw[-,line width=1pt] (3) -- (5);
\draw[-,line width=1pt] (3) -- (6);
\draw[-,line width=1pt] (4) -- (5);
\draw[-,line width=1pt] (4) -- (6);
\draw[-,line width=1pt] (1) -- (7);
\draw[-,line width=1pt] (1) -- (9);
\draw[-,line width=1pt] (2) -- (8);
\draw[-,line width=1pt] (2) -- (9);
}
\end{tikzpicture}
\caption{}\label{Example3}
\end{figure}

\begin{theorem}\label{clutter_theorem}
Let $\mathcal{C}$ be an $d$-uniform clutter, and $q\geq 3$. The minimum distance of the edge code $C_{\mathcal{C}}$ is given by 
$$
\delta(C_{\mathcal{C}})=\begin{cases}
(q-2)^d(q-1)^{s-d}&\mbox{ if } d\leq s/2, \mbox{ and } \mathfrak{C}_{d-1}(J_{d\times 2})\subseteq \mathcal{C}\\
(q-2)^{s-d}(q-1)^{d}&\mbox{ if }s/2 < d < s, \mbox{ and } \mathfrak{C}_{s-d-1}(J_{(s-d)\times 2})\subseteq \mathcal{C^*}\\
\end{cases}
$$
\end{theorem}
\begin{proof}
We start with the case $d\leq s/2$, and  $\mathfrak{C}_{d-1}(J_{d\times 2})\subseteq \mathcal{C}$. Proposition  \ref{squarefree-affine} and Remark \ref{chi_zeros} implies that
\begin{align*}
\delta(C_{\mathcal{C}})&=\min\{|T\setminus V_{T}(f)|\:|\: f\in KE(\mathcal{C})\}\\
&=|T|-\max\{|V_{T}(f)|\:|\: f\in KE(\mathcal{C})\}\leq (q-2)^d(q-1)^{s-d}					
\end{align*} 
To show the equality we need to find a polynomial in $KE(\mathfrak{C}_{d-1}(J_{d\times 2}))$ with exactly $(q-1)^s-(q-2)^d(q-1)^{s-d}$. Consider the set of vertices $V(J_{d\times 2})=\{t_{i_1},\dots,t_{i_{2d}}\}$, and the following polynomial 
$$
f=(t_{i_1}-t_{i_2})\dots(t_{i_{2d-1}}-t_{i_{2d}}).
$$
We claim that $f\in KE(\mathfrak{C}_{d-1}(J_{d\times 2}))$. We prove this by induction on $d$. For the $d=2$, $f=(t_{i_1}-t_{i_2})(t_{i_{3}}-t_{i_{4}})=t_{i_{1}}t_{i_{3}}-t_{i_{1}}t_{i_{4}}-t_{i_{2}}t_{i_{3}}+t_{i_{2}}t_{i_{4}}$. Notices that the monomials in $f$ are exactly the path of length $d-1=1$ of $J_{2\times2}$, which proves the first step of the induction. Now assume that the claim is true for $r<d$. Every partite path of length $d-1$ in $J_{d\times2}$ should end or start from $t_{i_1}$ or $t_{i_2}$, then if we remove these two vertices from $J_{d\times2}$ we obtain that all the partite paths of length $d-2$ in $J_{d\times2}$ that do not contain $t_{i_1}$ and $t_{i_2}$ are partite paths of  $J_{(d-1)\times2}$, with $V(J_{(d-1)\times 2})=\{t_{i_3},\dots,t_{i_{2d}}\}$. Since these partite paths in $J_{(d-1)\times 2}$ corresponds to the monomials of
$$
g=(t_{i_3}-t_{i_4})\dots(t_{i_{2d-1}}-t_{i_{2d}}),
$$
induction hypothesis implies that $g\in KE(\mathfrak{C}_{d-1}(J_{d\times 2}))$. From this we conclude that the monomials in 
$$
f=t_{i_1}g-t_{i_2}g=(t_{i_1}-t_{i_2})g
$$
form the partite paths of length $d-1$ in $J_{d\times2}$. Thus, $f\in KE(\mathfrak{C}_{d-1}(J_{d\times 2}))$. Since, $\mathfrak{C}_{d-1}(J_{d\times 2})\subseteq \mathcal{C}$, we get that $f\in KE(\mathcal{C})$. Proposition \ref{zero_poly} implies that 
$$
|V_T(f)|=(q-1)^s-(q-2)^d(q-1)^{s-d}.
$$
Hence, we obtain that $\delta(C_{\mathcal{C}})= (q-2)^d(q-1)^{s-d}$.
For the other case, assume that $s/2 < d < s$, and $\mathfrak{C}_{s-d-1}(J_{(s-d)\times 2})\subseteq \mathcal{C^*}$. Notice that the clutter $\mathcal{C}^*$ is homogeneous of degree $s-d\leq s/2$. The first case of the proof implies that the minimum distance of  the code $C_{\mathcal{C}^*}$ is $(q-2)^{s-d}(q-1)^{d}$. Finally, from Lemma \ref{aster_code} we conclude the result.

\end{proof}

\begin{remark}\

\begin{itemize}
\item One insight behind the preceding theorem arises from the case of graphs when $d=2$. For $d\leq s/2$, the code achieves the minimum possible distance only if the graph contains a cycle of length $4$. This is because the polynomial $f=(t_{i_1}-t_{i_2})(t_{i_3}-t_{i_4})$ is a squarefree polynomial with the maximum number of zeros. It's noteworthy that the monomials of this polynomial correspond to partite paths of the multipartite graph $J_{2\times 2}$, which forms a cycle of length $4$. Thus, the rationale for employing the clutter of partite paths is to extend this phenomenon to higher values of $d>2$.

\item Notice that for a given $d$-hypersimplex $\mathcal{P}$, the set $V_{d}$—comprising all $t^{a}$ where $a\in\mathcal{P}\cap\mathbb{Z}$—forms a $d$-uniform clutter, consisting of all the square-free monomials of degree $d$. It includes either $\mathfrak{C}_{d-1}(J_{d\times 2})$ or $\mathfrak{C}_{s-d-1}(J_{d\times 2})$ depending on the value of $d$. Consequently, Theorem \ref{clutter_theorem} extends the formula for the minimum distance of the toric code over hypersimplices \cite[Theorem 4.5]{Jara_Pinto_Villa}.
\item Theorem \ref{clutter_theorem} also provides a method to identify non-equivalent codes with the same minimum distance. Consider a four-cycle graph $G$. Both the edge code $C_{G}$ and the toric code $C_{\mathcal{P}}(2)$ of degree $2$ share the same minimum distance, yet they are distinct codes due to their differing dimensions.
\end{itemize}

\end{remark}

We continue with the computation of the minimum distance of edge codes for another family of hypergraphs. This result can be consider as an extension of the calculation of the minimum distance of squarefree evaluation codes \cite[Theorem 5.5]{Jara_Pinto_Villa}

\begin{theorem}\label{relative_code}
Let  $1\leq d_2\leq d_1$ be two integers, and $q\geq 3$. If the hypergraph $\mathcal{H}$ contains all the edges of size $r$ with $d_2\leq r\leq d_1$, the minimum distance of the edge code is given by 
$$
\delta(\mathcal{C}_{\mathcal{H}}) = \begin{cases} (q-2)^{d_{1}}(q-1)^{s-d_1}\hbox{ if } d_1+d_2\leq s\\
(q-2)^{s-d_{2}}(q-1)^{d_2}\hbox{ if }d_1+d_2> s.\\
\end{cases}
$$

\end{theorem}
\begin{proof}
Assume that $d_1+d_2\leq s$. Notice that every polynomial $f$ in $KE(\mathcal{H})$ has degree at most $d_1$, then  as before by Proposition \ref{squarefree-affine} and Remark \ref{chi_zeros} we have that 
$$
\delta(\mathcal{C}_{\mathcal{H}})\geq (q-2)^{d_{1}}(q-1)^{s-d_1}.
$$
Now we need to find a polynomial $f$ with exactly $(q-1)^s-(q-2)^{d_1}(q-1)^{s-d_1}$ zeros. For this consider the polynomial 
$$
f=(t_1-t_2)\cdots(t_{2d_2-1}-t_{2d_2})(t_{2d_1+1}-1)\cdots(t_{d_1+d_2}-1)
$$
Notice $f$ has $d_1$ binomials, then Proposition \ref{zero_poly} implies that 
$$
|V_T(f)|=(q-1)^s-(q-2)^{d_1}(q-1)^{s-d_1}.
$$
Thus, we obtain that $\delta(C_{\mathcal{H}})=(q-2)^{d_1}(q-1)^{s-d_1}$.

Assume now that $s<d_1+d_2$. Notice that the hypergraph $\mathcal{H}^*$ contains all the edges of degree $r$ with $s-d_1\leq r\leq s-d_2$, then the first case of the proof implies that the minimum distance of  the code $C_{\mathcal{H}^*}$ is 
$$
(q-2)^{s-d_2}(q-1)^{s-(s-d_2)}=(q-2)^{s-d_2}(q-1)^{d_2}.
$$ 

Finally, from Lemma \ref{aster_code} we conclude the result.
\end{proof}
  
\begin{remark}
In the previous theorem for the case $d_1+d_2> s$ we actually can find a polynomial such that   its associated codeword has minimum distance $(q-2)^{s-d_{2}}(q-1)^{d_2}$. Such polynomial is 
\begin{align*}
f&=h_1\cdots h_{s-d_2}t_{i_1}\cdots t_{i_{(d_1+d_2-s)}}\\
&=(t_1-t_2)\cdots(t_{2(s-d_1)-1}-t_{2(s-d_1)})(t_{2(s-d_1)+1}-1)\cdots(t_{2s-(d_1+d_2)}-1)t_{i_1}\cdots t_{i_{d_1+d_2-s}},
\end{align*}
where $h_i=t_{2i-1}-t_{2i}$ for $i=1,\ldots,s-d_1$, and $h_i=t_{2(s-d_1)+i}-1$ for $i=1,\ldots,d_1-d_2$.
\end{remark}

In the following result, we investigate the relationship between an edge code and its dual by establishing the self-orthogonality of these codes.

%For a first draft it is okey. I need to modifive this thing. I am Here

\begin{theorem}\label{dual_code}
Given a hypergraph $\mathcal{H}$, the code $\mathcal{C}_{\mathcal{H}}$ is self-ortogonal.
\end{theorem}
\begin{proof}
Consider the following set of monomials $\mathcal{B}=\{t^{b_1},\dots,t^{b_k}\}$, such that 
$$
t^{b_i}=t_1^{b_{i,1}}\cdots t_s^{b_{i,s}}:=\prod_{t_j\in{\rm supp}(t^{a_i})}t_j^{q-2},
$$
$E(\mathcal{H})=\{t^{a_1},\dots,t^{a_k}\}$ is the monomial basis of $\mathcal{C}_{\mathcal{H}}$ generated by the edges of $\mathcal{H}$. Notices that the footprint of $I(T)$ is given by $\Delta_{\prec}(I(T))=\{t^c\:|\:c=(c_1,\dots,c_s),\: 0\leq c_i\leq q-2\}$. Thus, \cite[Proposition 7.2]{Lopez2020TheDO} implies that the algebraic dual of $\mathcal{C}_{\mathcal{H}}$ is $\mathcal{L}^{\perp}=K(\Delta_{\prec}(I(T))\setminus \mathcal{B})$, and from \cite[Corollary 7.3]{Lopez2020TheDO} we obtain  the dual of $\mathcal{C}_{\mathcal{H}}$ on $T$,
$$
\mathcal{C}_{\mathcal{H}}^{\perp}=(\mathcal{L}^{\perp})_T.
$$
Since $\Delta_{\prec}(I(T))\setminus \mathcal{B}$ contains all the square-free monomials we conclude that $\mathcal{C}_{\mathcal{H}}\subseteq\mathcal{C}_{\mathcal{H}}^{\perp}$.

\end{proof}

\begin{remark}
Theorem \ref{dual_code}, and the CSS construction \cite[Theorem 2.5]{relativehulls} imply that there is a quantum code with quantum distance equal to the minimum distance of $\mathcal{C}_{\mathcal{H}}$.
\end{remark}

%I am here 

\subsection{Hamming weights of an edge code of a tree}

%%%%%%%%%%%%%%%%%%%%%%%%%%%%%%%%%%%%%%%%%%%%%%%%%%%%%%%%%
%\section{The weights of edge codes of graphs.}\label{wieght_of_graphs}
%%%%%%%%%%%%%%%%%%%%%%%%%%%%%%%%%%%%%%%%%%%%%%%%%%%%%%%%%
Motivated by the weight distribution calculation of toric codes over a hypersimplex in degree one \cite[Theorem 4.5]{JayuLo}, we aim to extend this result to higher degrees.  In the case of edge code of a tree we can find a similar weight distribution for codes with squarefree monomials of degree greater than one.\\

To estimate the Hamming weights, we need first to study the affine variety definite for a polynomial in $KE(G)$, where $G$ is connected tree. The next two proposition give us an idea about its cardinality.

\begin{proposition}\cite[Proposition 4.4]{JayuLo}\label{weights-degree-one}
Let $r$ be an integer such that $2\leq r\leq s$. If $\alpha_{j_{1}},\dots,\alpha_{j_{r}}\in K^*$, and $f_{r}=\alpha_{j_{1}}t_{i_{1}}+\cdots+\alpha_{j_{r}}t_{i_{r}}$, then 
$$
|V_{T}(f_{r})|=(q-1)^{s-1}-(q-1)^{s-2}+\cdots+(-1)^{r}(q-1)^{s-(r-1)}.
$$
\end{proposition}

\begin{proposition}\label{weights-degree-two}
Let $G$ be a connected tree. For each $f\in KE(G)$ we obtain that 
$$
|V_{T}(f)|=(q-1)^{s-1}-(q-1)^{s-2}+\cdots+(-1)^{r}(q-1)^{s-(r-1)}.
$$
where $r$ is the number of monomials in $f$.
\end{proposition}
\begin{proof}

%rewrite this proof!

%Let $P_{n}$ a path of length $n-1$, then the number of zeros on $T$ of a polynomial in $KE(P_n)$ is $|V_{T}(f)|=(q-1)^{s-1}-(q-1)^{s-2}+\cdots+(-1)^{r}(q-1)^{s-(r-1)}$, where $r\leq n-1$ is the number of monomials in $f$. 
We prove this by induction on $r$. If $r=2$, there are two possible structures for $f$
$$
t_{j_1}(\alpha_{j_2}t_{j_2}-\alpha_{j_3}t_{j_3}),\quad \alpha_{j_2}t_{j_1}t_{j_2}-\alpha_{j_3}t_{j_3}t_{j_4}.
$$
It is not difficult to realize that the number of zeros on $T$ of each one of these polynomials is equal to $(q-1)^{s-1}$. This establish the first step of the induction. Now suppose that the statement is true for $r-1$, and let's prove for $r$. Given the graph structure of a tree, there are not cycles, then we have that $f$ contains a variable  $t_i$ that appears only in one monomial $t_it_j$ of $f$, so the equation $f=0$ is equivalent to a equation $t_i=\frac{h}{t_j}$, where the polynimial $h\in KE(G)$ contains all the others monomials in $f$, which describe edge of the tree $G$. The number of solution on $T$ of the equation $f=0$ is given by the number of possible non-zero values of the variable $t_i$, this is 
$$
|V_{T}(f)|=(q-1)^{s-1}-|V_{T^{*}}(h)|,
$$
where $T^{*}$ is the affine torus in $s-1$ dimensions. Since $h$ has $r-1$ monomials, then by induction hypothesis we obtain that 
$$
|V_{T}(f)|=(q-1)^{s-1}-((q-1)^{s-1-1}+\cdots+(-1)^{r-1}(q-1)^{s-1-(r-2)}),
$$
which is the result that we wanted.
\end{proof}

%I am here

After obtaining a characterization of polynomials in $KE(G)$ when $G$ is a tree, we are now prepared to present the weight distribution for edge codes of trees.

\begin{theorem}\label{theo-weight-dist}
Let $C_G$ be the edge code of connected tree $G$ . For $1\leq t\leq (s-1)/2,$ the $t$-th Hamming weight of $C_G$ is given by
$$
(q-1)^{s}-(q-1)^{s-1}+(q-1)^{s-2}+\cdots+(-1)^{2t+1}(q-1)^{s-(2t-1)}.
$$
\end{theorem}
\begin{proof}
From Proposition \ref{weights-degree-two},
$$
|V_{T}(f_{m})|=(q-1)^{s-1}-(q-1)^{s-2}+\cdots+(-1)^{m}(q-1)^{s-(m-1)},
$$
where $f_{m}=a_{1}t_{i_{1}}t_{j_{1}}+\cdots+a_{m}t_{i_{m}}t_{j_{m}}\in KE(G)$ is a homogeneous squarefree polynomial of degree two with monomials given by the edges of $G$. 
If $1\leq t\leq (s-1)/2$, notice that 
\begin{align*}
|V_{T}(f_{2(t+1)})|&=|V_{T}(f_{2t})|-(q-1)^{s-(2t)}+(q-1)^{s-(2t+1)}<|V_{T}(f_{2t})|,\\
|V_{T}(f_{2t+1})|&=|V_{T}(f_{2t})|-(q-1)^{s-(2t)}<|V_{T}(f_{2t})|, \text{ and }\\
|V_{T}(f_{2t+1})|&=|V_{T}(f_{2t})|-(q-1)^{s-(2t)}\\
&<|V_{T}(f_{2t})|-(q-1)^{s-(2t)}+(q-1)^{s-(2t+1)}-(q-1)^{s-(2t+2)}=|V_{T}(f_{2t+3})|.\\
\end{align*}
This implies that the weight distribution of $C_G$ is given by the even values of $m.$ In other words, the $t$-th Hamming weight is 
\begin{align*}
|T|-&\left( (q-1)^{s-1}+(q-1)^{s-2}+\cdots+(-1)^{2t}(q-1)^{s-(2t-1)}\right)\\
=&(q-1)^{s}-(q-1)^{s-1}+(q-1)^{s-2}+\cdots+(-1)^{2t+1}(q-1)^{s-(2t-1)}.
\end{align*}
Thus, we obtain the result.
\end{proof}

\section{Examples}\label{Examples}%I am here%%%%
%%%%%%%%%%%%%%%%%%%%%%%%%%%%%%%%%%%%%%%%%%%%%%%%%%%%%%%%%
In this section we present specific calculations for the minimum distance of edge codes of some special graphs. All these results can be implemented using the Coding Theory package for {\it Macaulay 2} \cite{M2,CodPack}.
\begin{example}
Consider the path $P_4$ with four vertices
\begin{figure}[H]
\centering
\begin{tikzpicture}[line width=.5pt,scale=0.75]
		\tikzstyle{every node}=[inner sep=1pt, minimum width=5.5pt] 
\tiny{
\node (1) at (-1,1){$\bullet$};
\node (4) at (1,1) {$\bullet$};
\node (2) at (-1,-1) {$\bullet$};
\node (3) at (1,-1) {$\bullet$};
\node at (-1,1.3){$1$};
\node at (1,1.3) {$4$};
\node at (-1,-1.3){$2$};
\node at (1,-1.3) {$3$};
\draw[-,line width=1pt] (1) to (2);
\draw[-,line width=1pt] (2) to (3);
\draw[-,line width=1pt] (3) -- (4);
}
\end{tikzpicture}
\caption{}\label{Example3}
\end{figure}
The set of edges of this graph is 
$$
E=\{t_1t_2,t_2t_3,t_3t_4\}. 
$$
The following table contains all different types of polynomials in $KE(P_4)$ and the number of zeros on the torus,

\begin{table}[H]
\begin{tabular}{ |c|c| } 
 \hline
 $f$ & $|V_T(f)|$ \\ 
 \hline
 $t_it_j$ & $0$ \\ 
 \hline
 $(t_i+t_j)t_k$ & $(q-1)^3$ \\ 
 \hline
 $t_it_j+t_kt_h$ & $(q-1)^3$ \\ 
 \hline
 $t_it_j+t_jt_k+t_kt_h$ & $(q-1)^2(q-2)$ \\ 
 \hline
\end{tabular}
\caption{}
\label{table_path}
\end{table}
In particular, consider the polynomial $f=t_it_j+t_kt_h$. A point $(P_1,P_2,P_3,P_4)\in V_{T}(f)$ if and only if it is a solution on the torus to the equation 
$$t_it_j+t_kt_h=0,$$
And this is equivalent to $t_{i}=-\cfrac{t_kt_h}{t_j}$. The number of solution of this last  equation on the torus is equal to number of possible non-zero values that the variables $t_k$, $t_h$, and $t_j$ can take. This number is equal to $(q-1)^3$. Remark \ref{chi_zeros} implies that the minimum distance of the code $C_{P_4}$ is equal to $(q-1)^4-(q-1)^3$. We point out that this code is out of the scope of Theorem \ref{clutter_theorem} because the graph does not contain a cycle of length $4$.
\end{example}

\begin{example}
Let's compute the minimum distance of a edge code of a cycle with five vertices.
\begin{figure}[H]
\centering
\begin{tikzpicture}[line width=.5pt,scale=0.75]
\tikzstyle{every node}=[inner sep=1pt, minimum width=5.5pt] 
\tiny{
\node (1) at (0,1){$\bullet$};
\node (2) at (1,0) {$\bullet$};
\node (3) at (.5,-1) {$\bullet$};
\node (4) at (-.5,-1){$\bullet$};
\node (5) at (-1,0) {$\bullet$};
%\node at (0,-1.7){${\delta(C_{G})=(q-2)(q-1)^4}$};
\node at (0,1.2){$1$};
\node at (1.2,0) {$2$};
\node at  (.5,-1.2){$3$};
\node at (-.5,-1.2){$4$};
\node at (-1.2,0){$5$};
\draw[-,line width=1pt] (1) to (2);
\draw[-,line width=1pt] (1) to (5);
\draw[-,line width=1pt] (2) -- (3);
\draw[-,line width=1pt] (3) -- (4);
\draw[-,line width=1pt] (4) -- (5);
}
\end{tikzpicture}
\caption{}\label{Example3}
\end{figure}
Similarly to the previous example we present a table with the different types of polynomials in $KE(G)$ and their respective number of zeros on the affine torus. 
\begin{table}[H]
\begin{tabular}{|c|c|c|} 
 \hline
 $f$ & Type&$|V_T(f)|$\\ 
 \hline
 $t_1t_2$ & Monomial & $0$\\ 
 \hline
 $(t_1+t_3)t_2$ & Path length $2$ & $(q-1)^4$\\ 
 \hline
 $t_1t_2+t_4t_5$ & Two not connected edges &$(q-1)^4$ \\ 
 \hline
 $t_1t_2+t_2t_3+t_3t_4$ &  Path length $3$ &$(q-1)^4-(q-1)^3$\\ 
  \hline
 $t_1t_2+t_2t_3+t_4t_5 $&  Path plus edge &$(q-1)^4-(q-1)^3$ \\ 
 \hline
  $t_1t_2+t_2t_3+t_3t_4+t_4t_5$ &  Path length $4$ &$(q-1)^4-(q-1)^3+(q-1)^2$\\ 
 \hline
   $t_1t_2+t_2t_3+t_3t_4+t_4t_5+t_5t_1$ & Cycle &$\leq (q-1)^4$\\ 
 \hline

\end{tabular}
\caption{}
\label{table_cycle}
\end{table}

The number of zeros of the polynomial $t_1t_2+t_2t_3+t_3t_4$ is equal to the number of solutions of the equation 
$$
t_1=\cfrac{-(t_2t_3+t_3t_4)}{t_2}=\cfrac{-t_3(t_2+t_4)}{t_2}
$$
on the affine torus. This number of solution is determined by the number of possible non-zero values that the variables $t_2$, $t_3$, $t_4$, and $t_5$ can take. With these possible values  we can compute the value of $t_1$ that give us a solution of the polynomial. Since we are on the torus, $t_1$ should be non-zero, so we can assume that $t_4\neq -t_2$. Then, the number of possible values of the variables $t_2$, $t_3$, and $t_5$ is $(q-1)$ and for $t_4$ is $(q-2)$. Thus, the number of solutions is $(q-1)^3(q-2)=(q-1)^4-(q-1)^3$.\\
Furthermore, we conclude that $\max_{f\in KE(G)}|V_T(f)|=(q-1)^4$, thus the minimum distance of the edge code is given by $\delta(C_G)=|T|-\max_{f\in KE(G)}|V_T(f)|=(q-1)^5-(q-1)^4=(q-1)^4(q-2)$.
%Delio.Change this formulation
\end{example}

\begin{example}

Following the ideas of the two previous examples, in Table \ref{five_vertices} we present the minimum distance of the edge code $C_G$ for all connected graphs $G$ with five vertices.
%\newpage
\begin{table}[H]
%\centering 
\begin{tabular}{|c|c|c|c|c|c|}
\cline{3-3}  
\multicolumn{1}{c}{}& &
\begin{tikzpicture}[line width=.5pt,scale=0.75]
\tikzstyle{every node}=[inner sep=1pt, minimum width=5.5pt] 
\tiny{
\node (1) at (0,1){$\bullet$};
\node (2) at (1,0) {$\bullet$};
\node (3) at (.5,-1) {$\bullet$};
\node (4) at (-.5,-1){$\bullet$};
\node (5) at (-1,0) {$\bullet$};
\node at (0,-1.7){${\delta(C_{G})=(q-2)^2(q-1)^3}$};
\node at (0,1.2){$1$};
\node at (1.2,0) {$2$};
\node at  (.5,-1.2){$3$};
\node at (-.5,-1.2){$4$};
\node at (-1.2,0){$5$};
\draw[-,line width=1pt] (1) to (2);
\draw[-,line width=1pt] (1) to (3);
\draw[-,line width=1pt] (1) -- (4);
\draw[-,line width=1pt] (1) -- (5);
\draw[-,line width=1pt] (2) -- (3);
\draw[-,line width=1pt] (2) -- (4);
\draw[-,line width=1pt] (2) -- (5);
\draw[-,line width=1pt] (3) -- (4);
\draw[-,line width=1pt] (3) -- (5);
\draw[-,line width=1pt] (4) -- (5);
}
\end{tikzpicture}\\
\hline \begin{tikzpicture}[line width=.5pt,scale=0.75]
\tikzstyle{every node}=[inner sep=1pt, minimum width=5.5pt] 
\tiny{
\node (1) at (0,1){$\bullet$};
\node (2) at (1,0) {$\bullet$};
\node (3) at (.5,-1) {$\bullet$};
\node (4) at (-.5,-1){$\bullet$};
\node (5) at (-1,0) {$\bullet$};
\node at (0,-1.7){${\delta(C_{G})=(q-2)(q-1)^4}$};
\node at (0,1.2){$1$};
\node at (1.2,0) {$2$};
\node at  (.5,-1.2){$3$};
\node at (-.5,-1.2){$4$};
\node at (-1.2,0){$5$};
\draw[-,line width=1pt] (1) to (2);
\draw[-,line width=1pt] (1) to (5);
\draw[-,line width=1pt] (2) -- (3);
\draw[-,line width=1pt] (3) -- (4);
\draw[-,line width=1pt] (4) -- (5);
}
\end{tikzpicture}&\begin{tikzpicture}[line width=.5pt,scale=0.75]
\tikzstyle{every node}=[inner sep=1pt, minimum width=5.5pt] 
\tiny{
\node (1) at (0,1){$\bullet$};
\node (2) at (1,0) {$\bullet$};
\node (3) at (.5,-1) {$\bullet$};
\node (4) at (-.5,-1){$\bullet$};
\node (5) at (-1,0) {$\bullet$};
\node at (0,-1.7){${\delta(C_{G})=(q-2)^2(q-1)^3}$};
\node at (0,1.2){$1$};
\node at (1.2,0) {$2$};
\node at  (.5,-1.2){$3$};
\node at (-.5,-1.2){$4$};
\node at (-1.2,0){$5$};
\draw[-,line width=1pt] (1) to (2);
\draw[-,line width=1pt] (1) to (3);
\draw[-,line width=1pt] (1) -- (4);
\draw[-,line width=1pt] (1) -- (5);
\draw[-,line width=1pt] (2) -- (3);
\draw[-,line width=1pt] (2) -- (5);
\draw[-,line width=1pt] (3) -- (4);
\draw[-,line width=1pt] (4) -- (5);
}
\end{tikzpicture}&
\begin{tikzpicture}[line width=.5pt,scale=0.75]
\tikzstyle{every node}=[inner sep=1pt, minimum width=5.5pt] 
\tiny{
\node (1) at (0,1){$\bullet$};
\node (2) at (1,0) {$\bullet$};
\node (3) at (.5,-1) {$\bullet$};
\node (4) at (-.5,-1){$\bullet$};
\node (5) at (-1,0) {$\bullet$};
\node at (0,-1.7){${\delta(C_{G})=(q-2)^2(q-1)^3}$};
\node at (0,1.2){$1$};
\node at (1.2,0) {$2$};
\node at  (.5,-1.2){$3$};
\node at (-.5,-1.2){$4$};
\node at (-1.2,0){$5$};
\draw[-,line width=1pt] (1) to (3);
\draw[-,line width=1pt] (1) -- (5);
\draw[-,line width=1pt] (2) -- (3);
\draw[-,line width=1pt] (2) -- (5);
\draw[-,line width=1pt] (3) -- (4);
\draw[-,line width=1pt] (4) -- (5);
}
\end{tikzpicture}&\begin{tikzpicture}[line width=.5pt,scale=0.75]
\tikzstyle{every node}=[inner sep=1pt, minimum width=5.5pt] 
\tiny{
\node (1) at (0,1){$\bullet$};
\node (2) at (1,0) {$\bullet$};
\node (3) at (.5,-1) {$\bullet$};
\node (4) at (-.5,-1){$\bullet$};
\node (5) at (-1,0) {$\bullet$};
\node at (0,-1.7){${\delta(C_{G})=(q-2)^2(q-1)^3}$};
\node at (0,1.2){$1$};
\node at (1.2,0) {$2$};
\node at  (.5,-1.2){$3$};
\node at (-.5,-1.2){$4$};
\node at (-1.2,0){$5$};
\draw[-,line width=1pt] (1) to (2);
\draw[-,line width=1pt] (1) -- (5);
\draw[-,line width=1pt] (2) -- (3);
\draw[-,line width=1pt] (2) -- (4);
\draw[-,line width=1pt] (2) -- (5);
\draw[-,line width=1pt] (3) -- (5);
\draw[-,line width=1pt] (4) -- (5);
}
\end{tikzpicture}&\begin{tikzpicture}[line width=.5pt,scale=0.75]
\tikzstyle{every node}=[inner sep=1pt, minimum width=5.5pt] 
\tiny{
\node (1) at (0,1){$\bullet$};
\node (2) at (1,0) {$\bullet$};
\node (3) at (.5,-1) {$\bullet$};
\node (4) at (-.5,-1){$\bullet$};
\node (5) at (-1,0) {$\bullet$};
\node at (0,-1.7){${\delta(C_{G})=(q-2)^2(q-1)^3}$};
\node at (0,1.2){$2$};
\node at (1.2,0) {$1$};
\node at  (.5,-1.2){$3$};
\node at (-.5,-1.2){$4$};
\node at (-1.2,0){$5$};
\draw[-,line width=1pt] (1) to (2);
\draw[-,line width=1pt] (1) to (3);
\draw[-,line width=1pt] (1) -- (4);
\draw[-,line width=1pt] (1) -- (5);
\draw[-,line width=1pt] (2) -- (3);
\draw[-,line width=1pt] (2) -- (4);
\draw[-,line width=1pt] (3) -- (4);
\draw[-,line width=1pt] (3) -- (5);
\draw[-,line width=1pt] (4) -- (5);
}
\end{tikzpicture}\\
\hline 
\begin{tikzpicture}[line width=.5pt,scale=0.75]
\tikzstyle{every node}=[inner sep=1pt, minimum width=5.5pt] 
\tiny{
\node (1) at (0,1){$\bullet$};
\node (2) at (1,0) {$\bullet$};
\node (3) at (.5,-1) {$\bullet$};
\node (4) at (-.5,-1){$\bullet$};
\node (5) at (-1,0) {$\bullet$};
\node at (0,-1.7){${\delta(C_{G})=(q-2)(q-1)^4}$};
\node at (0,1.2){$1$};
\node at (1.2,0) {$2$};
\node at  (.5,-1.2){$3$};
\node at (-.5,-1.2){$4$};
\node at (-1.2,0){$5$};
\draw[-,line width=1pt] (1) to (2);
\draw[-,line width=1pt] (1) to (3);
\draw[-,line width=1pt] (1) -- (4);
\draw[-,line width=1pt] (1) -- (5);
}
\end{tikzpicture}&\begin{tikzpicture}[line width=.5pt,scale=0.75]
\tikzstyle{every node}=[inner sep=1pt, minimum width=5.5pt] 
\tiny{
\node (1) at (0,1){$\bullet$};
\node (2) at (1,0) {$\bullet$};
\node (3) at (.5,-1) {$\bullet$};
\node (4) at (-.5,-1){$\bullet$};
\node (5) at (-1,0) {$\bullet$};
\node at (0,-1.7){${\delta(C_{G})=(q-2)^2(q-1)^3}$};
\node at (0,1.2){$1$};
\node at (1.2,0) {$5$};
\node at  (.5,-1.2){$4$};
\node at (-.5,-1.2){$3$};
\node at (-1.2,0){$2$};
\draw[-,line width=1pt] (1) -- (5);
\draw[-,line width=1pt] (2) -- (3);
\draw[-,line width=1pt] (2) -- (4);
\draw[-,line width=1pt] (2) -- (5);
\draw[-,line width=1pt] (3) -- (4);
\draw[-,line width=1pt] (3) -- (5);
\draw[-,line width=1pt] (4) -- (5);
}
\end{tikzpicture}&\begin{tikzpicture}[line width=.5pt,scale=0.75]
\tikzstyle{every node}=[inner sep=1pt, minimum width=5.5pt] 
\tiny{
\node (1) at (0,1){$\bullet$};
\node (2) at (1,0) {$\bullet$};
\node (3) at (.5,-1) {$\bullet$};
\node (4) at (-.5,-1){$\bullet$};
\node (5) at (-1,0) {$\bullet$};
\node at (0,-1.7){${\delta(C_{G})=(q-2)^2(q-1)^3}$};
\node at (0,1.2){$1$};
\node at (1.2,0) {$3$};
\node at  (.5,-1.2){$4$};
\node at (-.5,-1.2){$5$};
\node at (-1.2,0){$2$};
\draw[-,line width=1pt] (1) to (2);
\draw[-,line width=1pt] (1) -- (5);
\draw[-,line width=1pt] (2) -- (3);
\draw[-,line width=1pt] (2) -- (4);
\draw[-,line width=1pt] (2) -- (5);
\draw[-,line width=1pt] (3) -- (4);
\draw[-,line width=1pt] (3) -- (5);
\draw[-,line width=1pt] (4) -- (5);
}
\end{tikzpicture}&
\begin{tikzpicture}[line width=.5pt,scale=0.75]
\tikzstyle{every node}=[inner sep=1pt, minimum width=5.5pt] 
\tiny{
\node (1) at (0,1){$\bullet$};
\node (2) at (1,0) {$\bullet$};
\node (3) at (.5,-1) {$\bullet$};
\node (4) at (-.5,-1){$\bullet$};
\node (5) at (-1,0) {$\bullet$};
\node at (0,-1.7){${\delta(C_{G})=(q-2)(q-1)^4}$};
\node at (0,1.2){$1$};
\node at (1.2,0) {$2$};
\node at  (.5,-1.2){$3$};
\node at (-.5,-1.2){$4$};
\node at (-1.2,0){$5$};
\draw[-,line width=1pt] (1) to (2);
\draw[-,line width=1pt] (1) -- (5);
\draw[-,line width=1pt] (1) -- (3);
\draw[-,line width=1pt] (4) -- (5);
}
\end{tikzpicture}&\begin{tikzpicture}[line width=.5pt,scale=0.75]
\tikzstyle{every node}=[inner sep=1pt, minimum width=5.5pt] 
\tiny{
\node (1) at (0,1){$\bullet$};
\node (2) at (1,0) {$\bullet$};
\node (3) at (.5,-1) {$\bullet$};
\node (4) at (-.5,-1){$\bullet$};
\node (5) at (-1,0) {$\bullet$};
\node at (0,-1.7){${\delta(C_{G})=(q-2)(q-1)^4}$};
\node at (0,1.2){$2$};
\node at (1.2,0) {$1$};
\node at  (.5,-1.2){$5$};
\node at (-.5,-1.2){$4$};
\node at (-1.2,0){$3$};
\draw[-,line width=1pt] (1) to (2);
\draw[-,line width=1pt] (1) -- (5);
\draw[-,line width=1pt] (2) -- (5);
\draw[-,line width=1pt] (3) -- (4);
\draw[-,line width=1pt] (3) -- (5);
\draw[-,line width=1pt] (4) -- (5);
}
\end{tikzpicture}\\
\hline \begin{tikzpicture}[line width=.5pt,scale=0.75]
\tikzstyle{every node}=[inner sep=1pt, minimum width=5.5pt] 
\tiny{
\node (1) at (0,1){$\bullet$};
\node (2) at (1,0) {$\bullet$};
\node (3) at (.5,-1) {$\bullet$};
\node (4) at (-.5,-1){$\bullet$};
\node (5) at (-1,0) {$\bullet$};
\node at (0,-1.7){${\delta(C_{G})=(q-2)^2(q-1)^3}$};
\node at (0,1.2){$1$};
\node at (1.2,0) {$3$};
\node at  (.5,-1.2){$5$};
\node at (-.5,-1.2){$4$};
\node at (-1.2,0){$2$};
\draw[-,line width=1pt] (1) to (2);
\draw[-,line width=1pt] (1) -- (5);
\draw[-,line width=1pt] (2) -- (3);
\draw[-,line width=1pt] (2) -- (4);
\draw[-,line width=1pt] (2) -- (5);
\draw[-,line width=1pt] (3) -- (4);
\draw[-,line width=1pt] (4) -- (5);
}
\end{tikzpicture}&
\begin{tikzpicture}[line width=.5pt,scale=0.75]
\tikzstyle{every node}=[inner sep=1pt, minimum width=5.5pt] 
\tiny{
\node (1) at (0,1){$\bullet$};
\node (2) at (1,0) {$\bullet$};
\node (3) at (.5,-1) {$\bullet$};
\node (4) at (-.5,-1){$\bullet$};
\node (5) at (-1,0) {$\bullet$};
\node at (0,-1.7){${\delta(C_{G})=(q-2)(q-1)^4}$};
\node at (0,1.2){$1$};
\node at (1.2,0) {$2$};
\node at  (.5,-1.2){$3$};
\node at (-.5,-1.2){$4$};
\node at (-1.2,0){$5$};
\draw[-,line width=1pt] (1) -- (5);
\draw[-,line width=1pt] (2) -- (5);
\draw[-,line width=1pt] (3) -- (4);
\draw[-,line width=1pt] (3) -- (5);
\draw[-,line width=1pt] (4) -- (5);
}
\end{tikzpicture}&\begin{tikzpicture}[line width=.5pt,scale=0.75]
\tikzstyle{every node}=[inner sep=1pt, minimum width=5.5pt] 
\tiny{
\node (1) at (0,1){$\bullet$};
\node (2) at (1,0) {$\bullet$};
\node (3) at (.5,-1) {$\bullet$};
\node (4) at (-.5,-1){$\bullet$};
\node (5) at (-1,0) {$\bullet$};
\node at (0,-1.7){${\delta(C_{G})=(q-2)(q-1)^4}$};
\node at (0,1.2){$2$};
\node at (1.2,0) {$1$};
\node at  (.5,-1.2){$3$};
\node at (-.5,-1.2){$5$};
\node at (-1.2,0){$4$};
\draw[-,line width=1pt] (1) -- (5);
\draw[-,line width=1pt] (2) -- (3);
\draw[-,line width=1pt] (3) -- (4);
\draw[-,line width=1pt] (3) -- (5);
\draw[-,line width=1pt] (4) -- (5);
}
\end{tikzpicture}&\begin{tikzpicture}[line width=.5pt,scale=0.75]
\tikzstyle{every node}=[inner sep=1pt, minimum width=5.5pt] 
\tiny{
\node (1) at (0,1){$\bullet$};
\node (2) at (1,0) {$\bullet$};
\node (3) at (.5,-1) {$\bullet$};
\node (4) at (-.5,-1){$\bullet$};
\node (5) at (-1,0) {$\bullet$};
\node at (0,-1.7){${\delta(C_{G})=(q-2)^2(q-1)^3}$};
\node at (0,1.2){$5$};
\node at (1.2,0) {$3$};
\node at  (.5,-1.2){$1$};
\node at (-.5,-1.2){$2$};
\node at (-1.2,0){$4$};
\draw[-,line width=1pt] (1) -- (5);
\draw[-,line width=1pt] (2) -- (3);
\draw[-,line width=1pt] (2) -- (4);
\draw[-,line width=1pt] (2) -- (5);
\draw[-,line width=1pt] (3) -- (4);
\draw[-,line width=1pt] (4) -- (5);
}
\end{tikzpicture}&
\begin{tikzpicture}[line width=.5pt,scale=0.75]
\tikzstyle{every node}=[inner sep=1pt, minimum width=5.5pt] 
\tiny{
\node (1) at (0,1){$\bullet$};
\node (2) at (1,0) {$\bullet$};
\node (3) at (.5,-1) {$\bullet$};
\node (4) at (-.5,-1){$\bullet$};
\node (5) at (-1,0) {$\bullet$};
\node at (0,-1.7){${\delta(C_{G})=(q-2)^2(q-1)^3}$};
\node at (0,1.2){$1$};
\node at (1.2,0) {$2$};
\node at  (.5,-1.2){$3$};
\node at (-.5,-1.2){$4$};
\node at (-1.2,0){$5$};
\draw[-,line width=1pt] (1) -- (5);
\draw[-,line width=1pt] (2) -- (3);
\draw[-,line width=1pt] (2) -- (5);
\draw[-,line width=1pt] (3) -- (4);
\draw[-,line width=1pt] (3) -- (5);
\draw[-,line width=1pt] (4) -- (5);
}
\end{tikzpicture}\\
\hline \begin{tikzpicture}[line width=.5pt,scale=0.75]
\tikzstyle{every node}=[inner sep=1pt, minimum width=5.5pt] 
\tiny{
\node (1) at (0,1){$\bullet$};
\node (2) at (1,0) {$\bullet$};
\node (3) at (.5,-1) {$\bullet$};
\node (4) at (-.5,-1){$\bullet$};
\node (5) at (-1,0) {$\bullet$};
\node at (0,-1.7){${\delta(C_{G})=(q-2)(q-1)^4}$};
\node at (0,1.2){$3$};
\node at (1.2,0) {$2$};
\node at  (.5,-1.2){$1$};
\node at (-.5,-1.2){$5$};
\node at (-1.2,0){$4$};
\draw[-,line width=1pt] (1) to (2);
\draw[-,line width=1pt] (1) -- (5);
\draw[-,line width=1pt] (2) -- (3);
\draw[-,line width=1pt] (4) -- (5);
}
\end{tikzpicture}&\begin{tikzpicture}[line width=.5pt,scale=0.75]
\tikzstyle{every node}=[inner sep=1pt, minimum width=5.5pt] 
\tiny{
\node (1) at (0,1){$\bullet$};
\node (2) at (1,0) {$\bullet$};
\node (3) at (.5,-1) {$\bullet$};
\node (4) at (-.5,-1){$\bullet$};
\node (5) at (-1,0) {$\bullet$};
\node at (0,-1.7){${\delta(C_{G})=(q-2)(q-1)^4}$};
\node at (0,1.2){$2$};
\node at (1.2,0) {$1$};
\node at  (.5,-1.2){$4$};
\node at (-.5,-1.2){$5$};
\node at (-1.2,0){$3$};
\draw[-,line width=1pt] (1) to (2);
\draw[-,line width=1pt] (1) -- (5);
\draw[-,line width=1pt] (3) -- (4);
\draw[-,line width=1pt] (3) -- (5);
\draw[-,line width=1pt] (4) -- (5);
}
\end{tikzpicture}
&\begin{tikzpicture}[line width=.5pt,scale=0.75]
\tikzstyle{every node}=[inner sep=1pt, minimum width=5.5pt] 
\tiny{
\node (1) at (0,1){$\bullet$};
\node (2) at (1,0) {$\bullet$};
\node (3) at (.5,-1) {$\bullet$};
\node (4) at (-.5,-1){$\bullet$};
\node (5) at (-1,0) {$\bullet$};
\node at (0,-1.7){${\delta(C_{G})=(q-2)^2(q-1)^3}$};
\node at (0,1.2){$1$};
\node at (1.2,0) {$2$};
\node at  (.5,-1.2){$3$};
\node at (-.5,-1.2){$4$};
\node at (-1.2,0){$5$};
\draw[-,line width=1pt] (1) to (2);
\draw[-,line width=1pt] (1) to (3);
\draw[-,line width=1pt] (1) -- (5);
\draw[-,line width=1pt] (2) -- (3);
\draw[-,line width=1pt] (2) -- (5);
\draw[-,line width=1pt] (3) -- (4);
\draw[-,line width=1pt] (4) -- (5);
}
\end{tikzpicture}&\begin{tikzpicture}[line width=.5pt,scale=0.75]
\tikzstyle{every node}=[inner sep=1pt, minimum width=5.5pt] 
\tiny{
\node (1) at (0,1){$\bullet$};
\node (2) at (1,0) {$\bullet$};
\node (3) at (.5,-1) {$\bullet$};
\node (4) at (-.5,-1){$\bullet$};
\node (5) at (-1,0) {$\bullet$};
\node at (0,-1.7){${\delta(C_{G})=(q-2)^2(q-1)^3}$};
\node at (0,1.2){$1$};
\node at (1.2,0) {$2$};
\node at  (.5,-1.2){$4$};
\node at (-.5,-1.2){$5$};
\node at (-1.2,0){$3$};
\draw[-,line width=1pt] (1) to (2);
\draw[-,line width=1pt] (1) -- (5);
\draw[-,line width=1pt] (2) -- (3);
\draw[-,line width=1pt] (2) -- (5);
\draw[-,line width=1pt] (3) -- (4);
\draw[-,line width=1pt] (4) -- (5);
}
\end{tikzpicture}&\begin{tikzpicture}[line width=.5pt,scale=0.75]
\tikzstyle{every node}=[inner sep=1pt, minimum width=5.5pt] 
\tiny{
\node (1) at (0,1){$\bullet$};
\node (2) at (1,0) {$\bullet$};
\node (3) at (.5,-1) {$\bullet$};
\node (4) at (-.5,-1){$\bullet$};
\node (5) at (-1,0) {$\bullet$};
\node at (0,-1.7){${\delta(C_{G})=(q-2)^2(q-1)^3}$};
\node at (0,1.2){$1$};
\node at (1.2,0) {$2$};
\node at  (.5,-1.2){$5$};
\node at (-.5,-1.2){$4$};
\node at (-1.2,0){$3$};
\draw[-,line width=1pt] (1) -- (5);
\draw[-,line width=1pt] (2) -- (3);
\draw[-,line width=1pt] (2) -- (5);
\draw[-,line width=1pt] (3) -- (4);
\draw[-,line width=1pt] (4) -- (5);
}
\end{tikzpicture}\\
\hline
\end{tabular}
\caption{}
\label{five_vertices}
\end{table}

Notice that the edge code $C_G$ has minimum distance $\delta(C_{G})=(q-2)^2(q-1)^3$ only when the graph $G$ contains a cycle of length $4$ or equivalent in terms of Theorem \ref{clutter_theorem} when $\mathfrak{C}_{1}(J_{2\times2})\subseteq G$.

\end{example}

\bibliographystyle{alpha}
\bibliography{References}

\end{document}